\documentclass[11pt]{article}

\usepackage{amsmath,amsfonts,amsthm,fullpage,palatino,eulervm,mathtools,xcolor}

\numberwithin{equation}{section}

\newtheorem{theorem}{Theorem}[section]

\newtheorem{corollary}[theorem]{Corollary}

{
\theoremstyle{definition}
\newtheorem{definition}[theorem]{Definition}
\newtheorem{example}[theorem]{Example}

}
{
\theoremstyle{remark}
\newtheorem{remark}[theorem]{Remark}
}

\DeclareMathOperator{\dar}{dar}

\DeclareMathOperator{\sech}{sech}

\usepackage[colorlinks=true,
linkcolor=webgreen,
filecolor=webbrown,
citecolor=webgreen,
]{hyperref}

\definecolor{webgreen}{rgb}{0,.5,0}
\definecolor{webbrown}{rgb}{.6,0,0}

\newcommand{\seqnum}[1]{\href{https://oeis.org/#1}{\rm \underline{#1}}}

\usepackage[symbol]{footmisc}

\makeatletter
\newcommand{\subjclass}[2][2020]{%
  \let\@oldtitle\@title%
  \gdef\@title{\@oldtitle\footnotetext{#1 \emph{Mathematics subject classification:} #2.}}%
}
\newcommand{\keywords}[1]{%
  \let\@@oldtitle\@title%
  \gdef\@title{\@@oldtitle\footnotetext{\emph{Key words and phrases:} #1.}}%
}
\makeatother

\begin{document}

\title{Pseudo-involutions in the Riordan group\\and Chebyshev polynomials}

\author{
\textsc{Alexander Burstein}\\
Department of Mathematics\\
Howard University\\
Washington, DC 20059, USA\\
\texttt{aburstein@howard.edu} 
\and 
\textsc{Louis W. Shapiro}\\
Department of Mathematics\\
Howard University\\
Washington, DC 20059, USA\\
\texttt{lshapiro@howard.edu}
}

\date{January 31, 2025}



\maketitle

\begin{center}
\emph{Dedicated to the memory of Emanuele Munarini and Renzo Sprugnoli}
\end{center}

\begin{abstract}
Generalizing the results in our previous paper, we consider pseudo-involutions in the Riordan group where the generating function $g$ for the first column of a Riordan array satisfies a functional equation of certain types involving a polynomial. For those types of equations, we find the pseudo-involutory companion of $g$. We also develop a general method for finding $B$-functions of Riordan pseudo-involutions in the cases we consider, and show that these $B$-functions involve Chebyshev polynomials. We apply our method for several families of Riordan arrays, obtaining new results and deriving known results more efficiently.
\end{abstract}

\section{Introduction} \label{sec:intro}

This article is a followup to our previous paper~\cite{BS} on pseudo-involutions in the Riordan group. This group, defined by Shapiro et al.~\cite{SGWW}, contains infinite lower triangular matrices called Riordan arrays, whose properties are generalize those of the Pascal triangle. A detailed introduction to the Riordan group can be found in books by Barry~\cite{Barry} and Shapiro et al.~\cite{SSBCHMW}, a survey by Davenport et al.~\cite{DFSW}, and a lecture series by Zeleke~\cite{Zeleke}, as well as in the original Shapiro et al.~\cite{SGWW} paper and several others. Our paper~\cite{BS} also contains an extended introduction, with a special emphasis on involutions and pseudo-involutions. Several of these sources contain extended bibliographies of the Riordan group literature as well. Therefore, we will contain ourselves here to a brief, focused introduction that lays the groundwork for the results discussed later in the paper.

We divide the rest of this introduction into two parts. In Subsection~\ref{subsec:basics}, we present a few basic facts about Riordan arrays needed in the rest of the paper. In Subsection~\ref{subsec:org}, we discuss the organization of the paper and its main results.

\subsection{Basics of Riordan arrays} \label{subsec:basics}

Consider the set $\mathcal{F}=\mathbb{R}[[z]]$ of formal power series with real coefficients (although many results would readily generalize for an arbitrary field $F$ of characteristic $0$ in place of $\mathbb{R}$). For a function $f(z)\in\mathcal{F}$ such that $f(z)=a_nz^n+a^{n+1}z^{n+1}+\dots$, $n\ge 0$, and $a_n\ne 0$, we will call $n$ the \emph{order} of $f(z)$. Observe that functions in $g=g(z)\in\mathcal{F}_0$ have a reciprocal $1/g$, whereas functions $f=f(z)\in\mathcal{F}_1$ have a compositional inverse that we denote $\overline{f}$ (so that $f\circ\overline{f}=\overline{f}\circ f=\mathrm{id}$, the identity function $\mathrm{id}(z)=z$).

Given an ordered pair $(g,f)\in\mathcal{F}_0\times\mathcal{F}_1$, we can define a lower triangular matrix $D=[d_{n,k}]_{n,k\ge 0}$ such that the $[d_{n,k}]^\mathrm{T}_{n\ge 0}$, the $k$-th column of $D$, has the generating function $gf^k$, i.e.
\[
\sum_{n\ge 0}d_{n,k}z^n=g(z)f(z)^k, \qquad k\ge 0.
\]
We denote this by $D=(g,f)$. Thus, for example, the Pascal triangle (as a lower triangular matrix) is $P=\left(\frac{1}{1-z},\frac{z}{1-z}\right)$. We call such a matrix $D$ a \emph{Riordan array}.

For a function $h=h(z)=\sum_{k\ge 0}h_kz^n\in\mathbb{F}$, the generating function of the product $D[h_n]^\mathbb{T}_{n\ge 0}$ is
\[
\sum_{n\ge 0}h_k(gf^k)=g\sum_{n\ge 0}h_kf^k=gh(f)=g\cdot(h\circ f).
\]
Shapiro et al.~\cite{SGWW} called this property the \emph{Fundamental Theorem of Riordan Arrays}. It implies that for any two Riordan arrays $(G,F)$ and $(g,f)$, we have
\[
(g,f)(G,F)=(g\cdot (G\circ f),F\circ f)=(gG(f),F(f)),
\]
so the identity matrix $(1,z)$ is the identity element
and
\[
(g,f)^{-1}=\left(\frac{1}{g(\overline{f})},\overline{f}\right)
\]
under the multiplication operation that corresponds to the matrix multiplication, and thus the Riordan arrays form a group called the \emph{Riordan group}.

Riordan group contains several interesting subgroups a few of which we will mention here. The \emph{$k$-Bell subgroup} contains arrays of the form $(g,zg^k)$ for a fixed $k$ (also called the \emph{Appell subgroup} for $k=0$ and the \emph{Bell subgroup} for $k=1$). Other standard examples include the \emph{checkerboard subgroup} of matrices $(g(z^2),zf(z^2))$, where $g,f\in\mathcal{F}_0$, $g$ is even, and $f$ is odd; the \emph{Lagrange}, \emph{derivative}, and \emph{hitting-time} subgroups of arrays $(g,f)$, where $g=1$, $g=f'$, and $g=zf'/f$, respectively, and $f\in\mathcal{F}_1$ is arbitrary.

A more generalized version of the Riordan group is as follows. Let $c=(c_n)_{n\ge 0}$ be a sequence of nonzero elements of $F$. Then we say that $D=(g,f)_c$ if
\[
\frac{g(z)f(z)^k}{c_k}=\sum_{n\ge 0}\frac{d_{n,k}}{c_n}z^n, \qquad k\ge 0.
\]
The two most frequently considered examples of this \emph{$c$-Riordan group} are the Riordan group itself, when $c_n=1$ for all $n\ge 0$, and the \emph{exponential Riordan group}, when $c_n=n!$ for all $n\ge 0$. Given $g\in\mathcal{F}_0$, $f\in\mathcal{F}_1$, we let $[g,f]$ be the exponential Riordan array corresponding to those two functions. For example, the Pascal triangle is also an exponential Riordan array $P=[e^z,z]$ in the exponential Appell subgroup of the exponential Riordan group.

Given our combinatorial motivation for considering Riordan arrays, we are interested in arrays $(g,f)$ containing only integer entries whose inverses also contain only integer entries. Therefore, we will assume from now on that $g(0)=1$, and $f'(0)=\pm1$, so that the main diagonal of $(g,f)$ contains either all $1$'s or alternating $1$'s and $-1$'s starting from $1$. Most of the arrays we consider will have $f'(0)=1$.

We call an Riordan array $(g,f)$ an \emph{involution} if $(g,f)^2=(g\cdot(g\circ f), f\circ f)=(1,z)$. A Riordan array is a \emph{pseudo-involution} if $(g,-f)=(g,f)(1,-z)$ is an involution. Equivalently, $(g,f)$ is a pseudo-involution if and only if $(g,f)^{-1}=(1,-z)(g,f)(1,-z)=(g(-z),-f(-z))$. Yet another equivalent condition for $(g,f)$ to be a pseudo-involution is that
\[
g\circ(-f)=\frac{1}{g}, \qquad \overline{f}=(-z)\circ f\circ (-z).
\]
When $(g,f)$ (or $[g,f]$) is a pseudo-involution, we call the function $f$ \emph{pseudo-involutory}. In this situation, $f$ is uniquely determined by $g$, so we call $f$ the \emph{pseudo-involutory companion} of $g$.

When $D=(g,f)$ is a pseudo-involution, the matrix $D$ can be recovered from the first column (the coefficient sequence for $g$) along (finite) antidiagonals rather than (infinite) columns, using what is called the \emph{B-function} of $f$, also called the B-function of the Riordan array $(g,f)$ (and originally called the $\Delta$-function by Cheon et al.~\cite{CJKS}).

\begin{definition} \label{def:b-fun}
The \emph{B-function} $B_f=B_f(z)$ of a Riordan pseudo-involution $D=(g,f)$ (and of its pseudo-involutory function $f$) is defined by the equation
\begin{equation} \label{eq:b-fun}
f-z=(zB_f)\circ(zf)=zf\,B_f(zf).
\end{equation}
The coefficient sequence $(b_n)_{n\ge 0}$ of $B_n(z)=\sum_{n\ge 0}b_nz^n\in\mathbb{R}[[z]]$ is called the \emph{B-sequence} of $(g,f)$ (or of $f$). 
\end{definition}

Equivalently, the B-sequence $(b_n)_{n\ge 0}$ of a pseudo-involution $D=[d_{n,k}]_{n,k\ge 0}$ can be defined directly by the recurrence relation
\begin{equation} \label{eq:b-mat}
d_{n+1,k+1}=d_{n,k}+\sum_{j=0}^{\infty}{b_j d_{n-j,k+j+1}}, \quad n,k\ge 0.
\end{equation}

For exponential Riordan arrays $[g,f]$ that are pseudo-involutions, the B-function $B_f$ is the same as for $(g,f)$, but we think of it an exponential generating function rather than an ordinary generating function, so we associate with it a related but different coefficient sequence.
\begin{definition} \label{def:beta-fun}
The \emph{beta-sequence} $(\beta_n)_{n\ge 0}$ of an exponential Riordan array $[g,f]$ (and of its pseudo-involutory function $f$) is given by $\beta_n=(2n+1)!b_n$ for $n\ge 0$. Equivalently, 
\[
B_f(z)=\sum_{n\ge 0}\frac{\beta_n}{(2n+1)!}z^n=\frac{Y(\sqrt{z})}{\sqrt{z}},
\]
where $Y(z)$ is the (odd) exponential generating function of $(0,\beta_0,0,\beta_1,0,\beta_2,\dots,0,\beta_n,\dots)$.
Defined directly, the $\beta$-sequence $(\beta_n)_{n\ge 0}$ of $D=[g,f]$ satisfies the recurrence relation 
\begin{equation} \label{eq:beta-mat}
\alpha_{n+1,k+1}=\alpha_{n,k}+\sum_{j=0}^{\infty}{\binom{n-k}{2j+1}\beta_j \alpha_{n-j,k+j+1}}, \quad n\ge k\ge 0,
\end{equation}
where $\alpha_{n,k}=\frac{d_{n,k}}{\binom{n}{k}}$ if $n\ge k$ and $\alpha_{n,k}=0$ if $n<k$.
\end{definition}

\subsection{Organization of the paper} \label{subsec:org}

In our previous paper~\cite{BS}, we have developed straightforward yet widely applicable methods for finding a pseudo-involutory companion $f$ of a B-function $g=g(z)$ and the B-function $B_f$ of $f=f(z)$ when $g$ satisfies a functional equation of the form $g=1+z\gamma(g)$ or $g=e^{z\gamma(g)}$ for a generalized palindrome $\gamma$ (see Definition~\ref{def:darga}) such that $\gamma(1)\ne 0$. Other connections between pseudo-involutions and palindromes have also been explored by He and Shapiro~\cite{HSh2}.

In Section~\ref{sec:b-h}, we take a wider view and consider several functions associated with Riordan arrays in general. We define the pseudo-inverse $\widehat{f}=(-z)\circ\overline{f}\circ(-z)$ of an invertible function $f$, and the pseudo-inverse $\widehat{(g,f)}=(1,-z)(g,f)^{-1}(1,-z)$ of a Riordan array, and consider the connection between pseudo-inverses and pseudo-involutions. In particular, it is straightforward to see that, for any Riordan array $X$, the array $X\widehat{X}$ is a pseudo-involution. In Theorem~\ref{thm:g-f-root}, we prove the converse of this result: if $(g,f)$ is a pseudo-involution, then $(g,f)=X\widehat{X}$ for some Riordan array $X$. Moreover, we describe the class of all Riordan arrays $X$ satisfying $X\widehat{X}=(g,f)$ and show that a canonical representative $X_{(g,f)}$ of this class. Likewise, we find the class of all invertible functions $h$ such that $f=h\circ\widehat{h}$ and describe the close relationship between the canonical representative of $h_f$ of that class (that we call the \emph{pseudo-half} of $f$) and the B-function $B_f$ of $f$. Finally, we give an alternative description of the set of all functions $g$ with a given pseudo-involutory companion $f$.

In Section~\ref{sec:gamma}, we extend these results by considering functions $\gamma$ that are not necessarily generalized palindromes to find the relationship between $\gamma$ associated with $g$ and the pseudo-half $h_f$ of the pseudo-involutory companion $f$ of $g$.

In Section~\ref{sec:b-from-gamma}, we consider $g$ satisfying either $g=1+z\gamma(g)$ or $g=e^{z\gamma(g)}$, where $\gamma$ is not necessarily a generalized palindrome, and express $B_f$ of the pseudo-involutory companion $f$ of $g$ as a composition of two functions $\eta$ and $H$, both of which are related to $\gamma$.

In Section~\ref{sec:poly-gamma}, we solve the problem in Section~\ref{sec:b-from-gamma} explicitly when $\gamma(z)$ is a Laurent polynomial (i.e. $\gamma(z)$ is allowed to contain terms with negative powers of $z$), and show that, in this case, both $\eta$ and $H$ are linear combinations of polynomials closely related to Chebyshev polynomials. In particular, we explicitly find $B_f$ for any quadratic polynomial $\gamma$, which includes the functional relations satisfied by Catalan, Schr\"{o}der, and extended Motzkin generating functions.

Finally, in Section~\ref{sec:rat-g}, we consider pseudo-involutions $(g,f)$ where $g(z)$ is a rational function of $z$. In this case, we find $B_f$ as a solution $y=B_f(z)$ of a bivariate polynomial equation $H_1(z,y)=H_2(z,y)$ for certain functions $H_1$ and $H_2$ related to the numerator and denominator of $g$, respectively, and involving Chebyshev polynomials.

\section{Functions associated with a Riordan array} \label{sec:b-h}

The following definition will be useful in the rest of the paper. 

\begin{definition} \label{def:hat}
If $f$ is an invertible power series, let $\bar{f}$ be its compositional inverse, i.e. $f\circ\bar{f}=\bar{f}\circ f=id$. Define the \emph{pseudo-inverse} $\hat{f}$ of $f$ by $\hat{f}=(-z)\circ\bar{f}\circ(-z)=\overline{(-z)\circ f\circ (-z)}$. Similarly, we call $\widehat{(g,f)}=(1,-z)(g,f)^{-1}(1,-z)$ the \emph{pseudo-inverse} of a Riordan array $(g,f)$. 
\end{definition}

It is easy to see that $\hat{\hat{f}}=f$, $\hat{\bar{f}}=\bar{\hat{f}}=(-z)\circ f \circ (-z)$, and $\widehat{f_1\circ f_2}=\hat{f}_2\circ\hat{f}_1$ for any invertible functions $f,f_1,f_2$. Moreover, $f$ is pseudo-involutory if and only if $\hat{f}=f$. Likewise, $(g,f)$ is a pseudo-involution if and only if $\widehat{(g,f)}=(g,f)$.

Given a function $h\in\mathbb{R}[[z]]$, we say that $h\in\mathcal{F}_k$, where $k\ge 0$, if $[x^k]h\ne 0$ and $[x^j]h=0$ for $0\le j<k$. Then for any invertible Riordan array $(g,f)$, we have $g\in\mathcal{F}_0$ and $f\in\mathcal{F}_1$. 

\begin{theorem} \label{thm:f-pseudo-half}
For any pseudo-involutory function $f\ne-z$, we have
\begin{equation} \label{eq:f-pseudo-half}
f=\widehat{(\sqrt{zf})}\circ\sqrt{zf}.
\end{equation}
Moreover, $f=h\circ\widehat{h}$  for $h\in\mathcal{F}_1$ if an only if $h=\widehat{(\sqrt{zf})}\circ \phi$, where $\phi$ is an arbitrary odd function, i.e. $\phi(-z)=-\phi(z)$.
\end{theorem}

\begin{proof}
Since $f\in\mathcal{F}_1$, it follows that $zf\in\mathcal{F}_2$, so $\pm\sqrt{zf}\in\mathcal{F}_1$ are invertible functions. Moreover, $(zf)\circ (-f) = (-f)\cdot(f\circ(-f))=(-f)\cdot(-z)=zf$. Therefore, $\sqrt{zf}\circ(-f)=\pm\sqrt{zf}$, and if $\sqrt{zf}\circ(-f)=\sqrt{zf}$, then $-f=\overline{\sqrt{zf}}\circ\sqrt{zf}=z$, i.e. $f=-z$, in which case $\sqrt{zf}\notin\mathbb{R}[[z]]$. Thus, if $f'(0)>0$ (and thus, $f'(0)=1$, since $f$ is pseudo-involutory), then $\sqrt{zf}\circ(-f)=-\sqrt{zf}$. In other words, $\sqrt{zf}\circ(-z)\circ f=(-z)\circ\sqrt{zf}$, or, equivalently,
\[
f=(-z)\circ\overline{\sqrt{zf}}\circ(-z)\circ\sqrt{zf}=\widehat{(\sqrt{zf})}\circ\sqrt{zf}.
\]
Now $h_1\circ\widehat{h}_1=h_2\circ\widehat{h}_2$ means $h_1\circ(-z)\circ\overline{h}_1\circ(-z)=h_2\circ(-z)\circ\overline{h}_2\circ(-z)$, which is equivalent to $\overline{h}_1\circ h_2\circ (-z)=(-z)\circ\overline{h}_1\circ h_2$. In other words, $\phi=\overline{h}_1\circ h_2$ is an odd function, i.e. $h_2=h_1\circ \phi$ for an arbitrary odd function $\phi$. Moreover, we know that $f=h\circ\widehat{h}$ for $h=\widehat{(\sqrt{zf})}$, so the second assertion of the theorem follows.
\end{proof}

Equivalently, we see that the involution $-f=\overline{\sqrt{zf}}\circ(-z)\circ\sqrt{zf}$, which implies that all involutions other than the identity are conjugate to $-z$. 

\begin{definition} \label{def:pseudo-half}
For a pseudo-involutory function $f$, denote $h_f:=\widehat{(\sqrt{zf})}$ and call $h_f$ the \emph{pseudo-half} of $f$.
\end{definition}

\begin{example} \label{ex:double-cat}
\emph{(Doubled Catalan sequence)}
Let $f=z(2C-1)$, where $C=1+zC^2=\frac{1-\sqrt{1-4z}}{2z}$ is the generating function for the Catalan sequence~\seqnum{A000108}.\footnote{Here and throughout the paper, the Axxxxxx number refers to the sequence number in the On-Line Encyclopedia of Integer Sequences (OEIS)~\cite{OEIS}.} Then we know (see, for example, \cite{BS}) that $f$ is pseudo-involutory. We can write $f=h_f\circ\widehat{h}_f$ for $h_f=\widehat{(\sqrt{zf})}=\widehat{(z\sqrt{2C-1})}$, but this is not the most convenient function $h$ such that $f=h\circ\widehat{h}$. It is much easier to write $f=z(2C-1)=z(C+zC^2)=zC+(zC)^2=(z+z^2)\circ(zC)= (z+z^2)\circ\overline{(z-z^2)}=(z+z^2)\circ\widehat{(z+z^2)}$, i.e. $f=h\circ\widehat{h}$ for $h=z+z^2$.

We can find the explicit formula for $h_f$ as follows. Recall that $zC=\overline{z-z^2}=\widehat{z+z^2}$. Then
\[
zf=z(2zC-z)=((z-z^2)(2z-(z-z^2))\circ(zC)=(z^2-z^4)\circ(zC),
\]
so $\sqrt{zf}=(z\sqrt{1-z^2})\circ(zC)$. Note that $z\sqrt{1-z^2}$ is an odd function of $z$, so $\widehat{(z\sqrt{1-z^2})}=\overline{z\sqrt{1-z^2}}=\overline{\sqrt{z}\circ(z-z^2)\circ z^2}=\sqrt{z}\circ(zC)\circ z^2=z\sqrt{C(z^2)}$, also an odd function. Thus,
\begin{equation} \label{eq:h-double-cat}
h_f=\widehat{(\sqrt{zf})}=\widehat{(z\sqrt{2C-1})}=\widehat{(zC)}\circ\widehat{(z\sqrt{1-z^2})}=(z+z^2)\circ\left(z\sqrt{C(z^2)}\right)=z\sqrt{C(z^2)}+z^2C(z^2).
\end{equation}
Moreover, we have $h=z+z^2=h_f\circ\widehat{(z\sqrt{C(z^2)})}=h_f\circ\overline{(z\sqrt{C(z^2)})}=h_f\circ(z\sqrt{1-z^2})$, and, as we noted above, $\phi(z)=z\sqrt{1-z^2}$ is an odd function.
\end{example}

For any function $h=h(z)\in\mathbb{R}[[z]]$ such that $h(0)=0$, we can write $h(z)=zh_o(z^2)+z^2h_e(z^2)$, where $h_o,h_e\in\mathbb{R}[[z]]$ are given by 
\begin{equation} \label{eq:def-he-ho}
h_o(z)=\frac{h(\sqrt{z})-h(-\sqrt{z})}{\sqrt{z}}, \qquad h_e(z)=\frac{h(\sqrt{z})+h(-\sqrt{z})}{z}.
\end{equation}
Then we have the following theorem.

\begin{theorem} \label{thm:b-half}
Let $f\ne-z$ be any pseudo-involutory function. Then 
\begin{equation} \label{eq:b-half}
B_f=2h_{f,e} \qquad \text{and} \qquad h_{f,o}=\sqrt{1+zh_{f,e}^2}.
\end{equation} 
Moreover, if $h$ is any invertible function such that $f=h\circ\widehat{h}$, then
\begin{equation} \label{eq:b-from-h}
zB_f=(2zh_e)\circ\overline{(zh_o^2-z^2h_e^2)}.
\end{equation}
\end{theorem}

\begin{proof}
Recall that $\hat{h}_f=\sqrt{zf}=\sqrt{z}\circ(zf)$. Therefore,
\[
\begin{split}
(zB_f)\circ(zf)&=f-z=h_f\circ\hat{h}_f-\overline{\hat{h}}_f\circ\hat{h}_f=(h_f-\overline{\hat{h}}_f)\circ\hat{h}_f=(h_f(z)-(-h_f(-z)))\circ\hat{h}_f\\
&=(h_f(z)+h_f(-z))\circ\hat{h}_f=(2z^2h_{f,e}(z^2))\circ\sqrt{zf}=(2zh_{f,e})\circ z^2\circ\sqrt{zf}=(2zh_{f,e})\circ(zf),
\end{split}
\]
so $zB_f=2zh_{f,e}$, and thus, $B_f=2h_{f,e}$. Furthermore,
\[
\begin{split}
z^2\circ\hat{h}_f
=z^2\circ\sqrt{zf}=zf=(\overline{\hat{h}}_f\circ\hat{h}_f)(h_f\circ\hat{h}_f)=(\overline{\hat{h}}_f h_f)\circ\hat{h}_f
=\left(-h_f(-z)h_f(z)\right)\circ\hat{h}_f,
\end{split}
\]
and therefore,
\[
\begin{split}
z^2&=-h_f(-z)h_f(z)=(zh_{f,o}(z^2)-z^2h_{f,e}(z^2))(zh_{f,o}(z^2)+z^2h_{f,e}(z^2))\\
&=z^2h_{f,o}^2(z^2)-z^4h_{f,e}^2(z^2)=(zh_{f,o}^2-z^2h_{f,e}^2)\circ z^2,
\end{split}
\]
so
$zh_{f,o}^2-z^2h_{f,e}^2=z$, or, equivalently, $h_{f,o}=\sqrt{1+zh_{f,e}^2}$.

Finally, suppose that $f=h\circ\hat{h}$ for some invertible function $h$. Then the equation $f-z=(zB_f)\circ(zf)$ can be rewritten as
\[
h\circ\hat{h}-\overline{\hat{h}}\circ\hat{h}=(zB_f)\circ(\overline{\hat{h}}h)\circ\hat{h}=(zB_f)\circ(-h(-z)h(z))\circ\hat{h},
\]
and therefore,
\[
(zB_f)\circ(-h(-z)h(z))=h-\overline{\hat{h}}=h(z)-(-h(-z))=h(z)+h(-z).
\]
In other words, as with $h_f$ earlier in this proof, we have
\[
(2zh_e)\circ z^2=(zB_f)\circ(zh_o^2-z^2h_e^2)\circ z^2,
\]
so
\[
2zh_e=(zB_f)\circ(zh_o^2-z^2h_e^2),
\]
and thus,
\[
zB_f=(2zh_e)\circ\overline{(zh_o^2-z^2h_e^2)}. \qedhere
\]
\end{proof}

\begin{corollary} \label{cor:b-half}
Let $f$ be any pseudo-involutory function. Then
\[
h_f=z\sqrt{1+\frac{z^2B_f^2(z^2)}{4}}+\frac{z^2B_f(z^2)}{2}=z\cdot\left(\left(\sqrt{1+\frac{z^2}{4}}+\frac{z}{2}\right)\circ\left(zB_f(z^2)\right)\right).
\]
Equivalently, $h_f=zu(zB_f(z^2))$, where $u=u(z)$ is such that $u(0)=1$, $u'(0)=1$, and $u-\dfrac{1}{u}=z$.
\end{corollary}

\begin{proof}
Immediate from Theorem~\ref{thm:b-half}.
\end{proof}

\begin{example} \label{ex:double-cat-hfb}
\emph{(Doubled Catalan sequence)}
Let $f=z(2C-1)$, the pseudo-involutory function from Example~\ref{ex:double-cat}. Then it follows from \eqref{eq:h-double-cat} that $h_{f,e}=C$ and $h_{f,o}=\sqrt{C}=\sqrt{1+zC^2}=\sqrt{1+zh_{f,e}^2}$. Therefore, we obtain $B_f=2h_{f,e}=2C$, as in \cite[Example 15]{BS}. 

From Example~\ref{ex:double-cat}, we see that we can also write $f=h\circ\widehat{h}$ for $h=z+z^2$. In that case, $h_e=1$ and $h_o=1$, so $2zh_e=2z$ and $zh_o^2-z^2h_e^2=z-z^2$. Therefore, from \eqref{eq:b-from-h}, $zB_f=(2zh_e)\circ\overline{(zh_o^2-z^2h_e^2)}=(2z)\circ\overline{(z-z^2)}=(2z)\circ(zC)=2zC$, so $B_f=2C$, as expected.
\end{example}

Finally, we can find all functions $g$ such that $(g,f)$ is a pseudo-involution. An answer to this problem was previously given in \cite{CK}:
\[
g=\pm\exp(\Phi(z,-f))
\] 
for any skew-symmetric bivariate function $\Phi(x,y)$, i.e. one such that $\Phi(y,x)=-\Phi(x,y)$. Our answer in this paper, though equivalent, has a somewhat simpler form.

\begin{theorem} \label{thm:g-from-zf}
Let $f$ be pseudo-involutory. Then a Riordan array $(g,f)$ is a pseudo-involution if and only if $g=\pm \exp(\phi(\sqrt{zf}))$, where $\phi=\phi(z)\in\mathbb{R}[[z]]$ is an arbitrary odd function.
\end{theorem}

\begin{proof}
Since $f(0)=0$ and $g(-f)=1/g$, it follows that $g(0)=1/g(0)$, and thus $g(0)=\pm 1$. Therefore, we may assume without loss of generality that $g(0)=1$ and show that $(g,f)$ is a pseudo-involution if and only if $g=\exp(\phi(\sqrt{zf}))$ and $\phi$ is odd.

Consider $G=G(z)$ such that $g=G\circ\sqrt{zf}$, i.e. $G=g\circ\overline{\sqrt{zf}}$. Then $g(-f)=G\circ\sqrt{zf}\circ(-f)=G\circ(-\sqrt{zf})=G(-z)\circ\sqrt{zf}$. Recall that, for a pseudo-involutory $f$, $(g,f)$ is a pseudo-involution if and only if $g\cdot g(-f)=1$. This condition is equivalent to
\[
1=(G\circ\sqrt{zf})\cdot(G(-z)\circ\sqrt{zf})=(G(z)G(-z))\circ\sqrt{zf},
\] 
or, in other words, $G(z)G(-z)=1$. Note that $\overline{\sqrt{zf}}$ is invertible, and thus $\overline{\sqrt{zf}}(0)=0$, so $G(0)=g(0)=1$, and therefore, $\log G$ is defined. Taking logarithms of both sides of $G(z)G(-z)=1$, we obtain $\log G(z)+\log G(-z)=0$, or, equivalently, $\log G(-z)=-\log G(z)$. Thus, $\log G=\phi$ for some odd function $\phi=\phi(z)$, i.e. $G=\exp(\phi)$. Moreover, $\phi(0)=\log G(0)=\log 1=0$, so $\phi\in\mathbb{R}[[z]]$. (This eliminates odd functions like $1/z$ that are not defined at $0$.) We can also check that, in this case, $G(z)G(-z)=1$ for any odd function $\phi\in\mathbb{R}[[z]]$. Therefore, $g\cdot g(-f)=1$ with $g(0)=1$ if and only if
\[
g=G\circ\sqrt{zf}=\exp(\phi(\sqrt{zf})),
\]
where $\phi\in\mathbb{R}[[z]]$ is an arbitrary odd function.
\end{proof}

Note that Theorem \ref{thm:g-from-zf} generalizes \cite[Corollary 2.4]{CK} that proved this result for $f=z$.

To see that the two representations of $G$ are equivalent, recall that $z=\overline{\sqrt{zf}}\circ\sqrt{zf}$ and $-f=\overline{\sqrt{zf}}\circ(-z)\circ\sqrt{zf}$, so
\[
\Phi(z,-f)=\Phi(\overline{\sqrt{zf}},\overline{\sqrt{zf}}(-z))\circ\sqrt{zf}, \qquad \Phi(-f,z)=\Phi(\overline{\sqrt{zf}}(-z),\overline{\sqrt{zf}})\circ\sqrt{zf},
\]
Then $\Phi(-f,z)=-\Phi(z,-f)$ implies that
\[
\Phi(\overline{\sqrt{zf}}(-z),\overline{\sqrt{zf}})=-\Phi(\overline{\sqrt{zf}},\overline{\sqrt{zf}}(-z)).
\]
Now let $\phi(z)=\Phi(\overline{\sqrt{zf}},\overline{\sqrt{zf}}(-z))$, to obtain $\phi(-z)=-\phi(z)$, as desired.

\bigskip

For a Riordan array $X=(G,F)$, let the \emph{pseudo-inverse} of $X$ be the Riordan array
\[
\hat{X}=(1,-z)X^{-1}(1,-z)=\left(\frac{1}{G}\circ\overline{F}\circ(-z),(-z)\circ\overline{F}\circ(-z)\right)
=\left(\frac{1}{G}\circ(-\hat{F}),\hat{F}\right).
\]
As for pseudo-involutory functions $f$, it is straightforward to see that $\widehat{XY}=\hat{Y}\hat{X}$, $\widehat{X^{-1}}=\hat{X}^{-1}=(1,-z)X(1,-z)$, and that $X$ is a pseudo-involution if and only if $\hat{X}=X$. Then Theorem \ref{thm:g-from-zf} also lets us reformulate and slightly strengthen \cite[Corollary 2.6]{CK} and \cite[Theorem 3.1]{CK} as follows. 
\begin{theorem} \label{thm:g-f-root}
A Riordan array $(g,f)$ is a pseudo-involution if and only if $(g,f)=X\hat{X}$, where $X$ is any Riordan array of the form $X=X_{(g,f)}\Psi$, where
\[
X_{(g,f)}=\left(\sqrt{g},\sqrt{zf}\right), \quad \Psi=\left(\psi_1(z^2),z\psi_2(z^2)\right)
\] 
for arbitrary $\psi_1,\psi_2\in\mathcal{F}_0$.
\end{theorem}

In other words, $\Psi$ is an arbitrary array in the \emph{checkerboard subgroup} of the Riordan group.

\begin{proof}
Let us note first that for any Riordan array $X$, the array $X\hat{X}$ is a pseudo-involution. Indeed,
\[
\widehat{X\hat{X}}=\hat{\hat{X}}\hat{X}=X\hat{X},
\]
so $X\hat{X}$ is a pseudo-involution. Moreover, Theorem~\ref{thm:f-pseudo-half} and the fact that $g\circ(-f)=1/g$ imply that
\[
\begin{split}
X_{(g,f)}\widehat{X_{(g,f)}}&=\left(\sqrt{g},\sqrt{zf}\right)\left(\frac{1}{\sqrt{g}}\circ(-\widehat{\sqrt{zf}}),\widehat{\sqrt{zf}}\right)\\
&=\left(\sqrt{g}\cdot\left(\frac{1}{\sqrt{g}}\circ(-\widehat{\sqrt{zf}})\circ\sqrt{zf}\right),\widehat{\sqrt{zf}}\circ\sqrt{zf}\right)\\
&=\left(\sqrt{g}\cdot\left(\frac{1}{\sqrt{g}}\circ(-f)\right),f\right)=\left(\sqrt{g}\cdot\frac{1}{\sqrt{1/g}},f\right)=(g,f).
\end{split}
\]
Finally, suppose that for some invertible Riordan arrays $X$ and $Y$, we have $X\hat{X}=Y\hat{Y}$. Then
\[
\Psi=X^{-1}Y=\hat{X}\hat{Y}^{-1}=\hat{X}\widehat{Y^{-1}}=\widehat{Y^{-1}X}=\widehat{\Psi^{-1}}=(1,-z)\Psi(1,-z),
\]
so that if $\Psi=(G,F)$, then $(G,F)=(1,-z)(G,F)(1,-z)=(G(-z),-F(-z))$, and therefore, $G$ is even and $F$ is odd. Moreover, $\Psi$ is also invertible, so $G\in\mathcal{F}_0$ and $F\in\mathcal{F}_1$, and thus $\Psi=\left(\psi_1(z^2),z\psi_2(z^2)\right)$ for some $\psi_1,\psi_2\in\mathcal{F}_0$, and $Y=X\Psi$. We also see that there are no other restrictions on $\psi_1,\psi_2$, and thus Riordan arrays of the form $X=X_{(g,f)}\Psi$ for an arbitrary checkerboard array $\Psi$ are exactly those for which $(g,f)=X\hat{X}$.
\end{proof}

\section{The $\gamma$ function of a Riordan array} \label{sec:gamma}

Many the results in the rest of this paper stem from two simple observations given in \cite[Theorem 2]{BS}, one more often useful in the case of ordinary generating functions, and the other more often useful in the case of exponential generating functions. We restate both here for completeness after a few preliminary remarks.

\begin{definition} \label{def:darga}
For a function $\gamma=\gamma(z)$, we say that $\gamma(z)$ has \emph{darga} $d$ if 
\[
\frac{\gamma(z)}{\gamma\left(\frac{1}{z}\right)}=z^d\phi(z)
\]
for some $\phi(z)\in\mathbb{R}[[z]]$ with $\phi(0)\ne 0$. In this case, we write $d=\dar(\gamma)$. In particular, if $\phi(z)=1$, we call $\gamma$ a \emph{generalized palindrome} of darga $d$. 
\end{definition}

We denote by $\mathcal{P}_d$ the set of all generalized palindromes of darga $d$.

For example, when $\gamma(z)=\sum_{n=k}^{l}a_nz^n$ is a Laurent polynomial, that is, $k,l\in\mathbb{Z}$, $k\le l$, and $a_k,a_l\ne 0$, it is straightforward to check that $\dar(\gamma)=k+l$, the sum of the minimum and maximum degrees of $\gamma$.

Darga (stress on the second syllable) has properties similar to degree. They are listed in \cite{BS} for generalized palindromes, and we briefly recall them here. If $\dar(\gamma_1)=d_1$ and $\dar(\gamma_2)=d_2$, then $\dar(\gamma_1\gamma_2)=d_1+d_2$, $\dar(\gamma_1/\gamma_2)=d_1-d_2$, and $\dar(\gamma^r(z))=\dar(\gamma(z^r))=r\dar(\gamma)$. Moreover, if both $\gamma_1$ and $\gamma_2$ are generalized palindromes, and $\dar(\gamma_1)=\dar(\gamma_2)=d$, then $\dar(\gamma_1+\gamma_2)=d$ when $\gamma_1+\gamma_2\ne 0$. Finally, if $\dar(\gamma_2)=0$, then $\dar(\gamma_1\circ\gamma_2)=0$.

\begin{theorem}[\cite{BS}, Theorem 2] \label{thm:f-companion} ~\\[-\baselineskip]
\begin{enumerate}
\item Let $g=g(z)$ be a function that satisfies a functional relation 
\begin{equation} \label{eq:gamma-ogf}
g=1+z\gamma(g)=1+z\cdot(\gamma\circ g),
\end{equation}
for some function $\gamma=\gamma(z)$ such that $\gamma(1)\ne 0$. Then $(g,f)$ and $[g,f]$ are pseudo-involutions if and only if
\begin{equation} \label{eq:f-ogf}
f=z\frac{\gamma(g)}{g\gamma\left(\frac{1}{g}\right)}.
\end{equation}
Moreover, if $\gamma\in\mathcal{P}_d$, then $f=zg^{d-1}$, so $(g,f)$ and $[g,f]$ are in the $(d-1)$-Bell subgroup.
\item Let $g=g(z)$ be a function that satisfies a functional relation 
\begin{equation} \label{eq:gamma-egf}
g=e^{z\gamma(g)}=e^{z\cdot(\gamma\circ g)},
\end{equation}
for some function $\gamma=\gamma(z)$ such that $\gamma(1)\ne 0$. Then $(g,f)$ and $[g,f]$ are pseudo-involutions if and only if
\begin{equation} \label{eq:f-egf}
f=z\frac{\gamma(g)}{\gamma\left(\frac{1}{g}\right)}.
\end{equation}
Moreover, if $\gamma\in\mathcal{P}_d$, then $f=zg^d$, so $(g,f)$ and $[g,f]$ are in the $d$-Bell subgroup.
\end{enumerate}
\end{theorem}

The following theorem relates the $\gamma$ function defined in each case above to the associated pseudo-involutory function $f$ in an equivalent but different way.

\begin{theorem} \label{thm:pseudo-half-gamma} ~\\[-\baselineskip]
\begin{enumerate}
\item Let $g=g(z)$ be a function that satisfies the functional relation \eqref{eq:gamma-ogf}. Then $(g,f)$ and $[g,f]$ are pseudo-involutions if and only if
\begin{equation} \label{eq:f-ogf-gamma}
f=\frac{z}{\gamma(1-z)}\circ\frac{z}{1+z}\circ\widehat{\frac{z}{\gamma(1-z)}},
\end{equation}
or, equivalently, if and only if $f=h\circ\widehat{h}$ for
\begin{equation} \label{eq:h-ogf-gamma}
h=\frac{z}{\gamma(1-z)}\circ\frac{2z}{2+z}=\frac{2z}{(2+z)\gamma\left(\frac{2-z}{2+z}\right)}
=2z\circ \frac{z}{(1+z)\gamma\left(\frac{1-z}{1+z}\right)}\circ\frac{z}{2}.
\end{equation}
\item Let $g=g(z)$ be a function that satisfies the functional relation \eqref{eq:gamma-egf}. Then $(g,f)$ and $[g,f]$ are pseudo-involutions if and only if
\begin{equation} \label{eq:f-egf-gamma}
f=\frac{z}{\gamma(e^{-z})}\circ\widehat{\frac{z}{\gamma(e^{-z})}},
\end{equation}
or, equivalently, if and only if $f=h\circ\widehat{h}$ for
\begin{equation} \label{eq:h-egf-gamma}
h=\frac{z}{\gamma(e^{-z})}.
\end{equation}
\end{enumerate}
\end{theorem}

\begin{proof}
Suppose $g$ satisfies the functional relation \eqref{eq:gamma-ogf}, then $g(0)=1$. Let $G=G(z)=g-1$, then $G(0)=0$ and $G$ satisfies the equation $G=z\gamma(1+G)$, i.e. $z=\frac{G}{\gamma(1+G)}$, or, equivalently,
\[
G=\overline{\left(\frac{z}{\gamma(1+z)}\right)}=\overline{(-z)\circ\frac{z}{\gamma(1-z)}\circ(-z)}=\widehat{\frac{z}{\gamma(1-z)}}.
\]
Therefore, by Theorem~\ref{thm:f-companion}, we have
\[
f=\frac{z\gamma(g)}{g\gamma(1/g)}=\frac{G}{(1+G)\gamma\left(\frac{1}{1+G}\right)}=\frac{z}{\gamma(1-z)}\circ\frac{G}{1+G}
=\frac{z}{\gamma(1-z)}\circ\frac{z}{1+z}\circ G,
\]
which implies Equation \eqref{eq:f-ogf-gamma}. Combined with the fact that $\frac{z}{1+z}=\frac{z}{1+\frac{z}{2}}\circ\frac{z}{1+\frac{z}{2}}=\frac{z}{1+\frac{z}{2}}\circ\widehat{\frac{z}{1+\frac{z}{2}}}$, this yields Equation~\eqref{eq:h-ogf-gamma}.

Now suppose $g$ satisfies the functional relation \eqref{eq:gamma-egf}, then $g(0)=1$. Let $G=G(z)=\log g$, then $G(0)=0$ and $G$ satisfies the equation $G=z\gamma(e^G)$, i.e. $z=\frac{G}{\gamma(e^G)}$, or, equivalently,
\[
G=\overline{\left(\frac{z}{\gamma(e^z)}\right)}=\overline{(-z)\circ\frac{z}{\gamma(e^{-z})}\circ(-z)}=\widehat{\frac{z}{\gamma(e^{-z})}}.
\]
Therefore, by Theorem~\ref{thm:f-companion}, we have
\[
f=\frac{z\gamma(g)}{\gamma(1/g)}=\frac{G}{\gamma(e^{-G})}=\frac{z}{\gamma(e^{-z})}\circ\widehat{\frac{z}{\gamma(e^{-z})}},
\]
in other words, Equations~\eqref{eq:f-egf-gamma} and~\eqref{eq:h-egf-gamma}.
\end{proof}

\section{$B$-function from $\gamma$-function} \label{sec:b-from-gamma}

In our previous paper \cite{BS}, we mostly considered cases where the function $\gamma(z)$ was a \emph{generalized palindrome}, i.e. when $\gamma(z)/\gamma(1/z)=z^k$ for some $k\in\mathbb{Z}$. Here we consider a wider collection of functions $\gamma=\gamma(z)$, not necessarily generalized palindromes, such that $\gamma(1)\ne 0$ and $\gamma'(1)$ exists. In some cases, such as for Laurent polynomials $\gamma$, we give more explicit formulas for the pseudo-involutory companion $f=f(z)$ of $g$ and the $B$-function $B_f$ of the pseudo-involution $(g,f)$.

Then we have the following two theorems.

\begin{theorem} \label{thm:bf-ogf}
Let $g$ be a function that satisfies Equation \eqref{eq:gamma-ogf}, and let $f=z\frac{\gamma(g)}{g\gamma(1/g)}$. Then $B_f=H\circ\overline{\left(\frac{z}{\eta}\right)}$, where $H=H(z)$ and $\eta=\eta(z)$ are defined by
\[
\eta\left(\frac{(z-1)^2}{z}\right)=\gamma(z)\gamma\left(\frac{1}{z}\right), \qquad H\left(\frac{(z-1)^2}{z}\right)=\frac{\gamma(z)-z\gamma\left(\frac{1}{z}\right)}{z-1}.
\]
\end{theorem}

\begin{proof}
Since $g=1+z\gamma(g)$, we have
\[
\frac{(g-1)^2}{g}=\frac{z^2\gamma(g)^2}{g}=z^2\frac{\gamma(g)}{g\gamma\left(\frac{1}{g}\right)}\gamma(g)\gamma\left(\frac{1}{g}\right)=(zf)\cdot\left(\left(\gamma(z)\gamma\left(\frac{1}{z}\right)\right)\circ g\right)
\]
The function $\gamma(z)\gamma(1/z)$ has darga $0$, since it is invariant under the substitution $z\mapsto \frac{1}{z}$. Therefore, $\gamma(z)\gamma(1/z)$ is a function of $z+\frac{1}{z}$, or, equivalently, of $z+\frac{1}{z}-2=\frac{(z-1)^2}{z}$. Hence, $\gamma(z)\gamma(1/z)=\eta\left(\frac{(z-1)^2}{z}\right)$ for some formal power series $\eta=\eta(z)$. Letting $z=1$, we obtain $\eta(0)=\gamma(1)^2\ne 0$, and thus the function $\frac{z}{\eta}$ has a compositional inverse as a formal power series. Therefore,
\[
\frac{(g-1)^2}{g} = (zf)\cdot \eta\left(\frac{(g-1)^2}{g}\right),
\]
or, equivalently,
\[
zf=\left(\frac{z}{\eta}\right)\circ\left(\frac{(g-1)^2}{g}\right),
\]
and thus,
\[
\frac{(g-1)^2}{g}=\overline{\left(\frac{z}{\eta}\right)}\circ(zf).
\]
On the other hand, we have 
\[
B_f(zf)=\frac{f-z}{zf}=\frac{\frac{f}{z}-1}{f}=\frac{\gamma(g)-g\gamma\left(\frac{1}{g}\right)}{z\gamma(g)}=\frac{\gamma(g)-g\gamma\left(\frac{1}{g}\right)}{g-1}=\frac{\gamma(z)-z\gamma\left(\frac{1}{z}\right)}{z-1}\circ g.
\]
Notice that $\frac{\gamma(z)-z\gamma\left(\frac{1}{z}\right)}{z-1}=\frac{\frac{1}{\sqrt{z}}\gamma(z)-\sqrt{z}\gamma\left(\frac{1}{z}\right)}{\sqrt{z}-\frac{1}{\sqrt{z}}}$ is also invariant under the substitution $z\mapsto\frac{1}{z}$ and, therefore, is a formal power series in $\frac{(z-1)^2}{z}$. Let $H=H(z)$ be defined by $H\left(\frac{(z-1)^2}{z}\right)=\frac{\gamma(z)-z\gamma(1/z)}{z-1}$. In particular, letting $z=1$, we see that the numerator of $H$ becomes $0$, so $1-z$ divides $\gamma(z)-z\gamma(1/z)$, and in fact,
\[
H(0)=\lim_{z\to 1}\frac{\gamma(z)-z\gamma\left(\frac{1}{z}\right)}{z-1}=\frac{d}{dz}\!\left.\left(\frac{\gamma(z)-z\gamma\left(\frac{1}{z}\right)}{z-1}\right)\right\vert_{z=1}=2\gamma(1)-\gamma'(1),
\]
so the power series expansion of $H(z)$ involves only nonnegative powers of $z$. Putting together all of the above, we obtain
\[
B_f(zf)=H\left(\frac{(g-1)^2}{g}\right) = H\circ \overline{\left(\frac{z}{\eta}\right)}\circ(zf),
\]
and thus
\[
B_f=H\circ\overline{\left(\frac{z}{\eta}\right)},
\]
as claimed.
\end{proof}

\begin{example} \label{ex:little-sch}
\emph{(Little Schr\"{o}der numbers)}
Let $r=r(z)$ and $s=s(z)$ be the generating functions for the large Schr\"{o}der numbers \seqnum{A006318} and the little Schr\"{o}der numbers \seqnum{A001003}, respectively. Then $r=1+z(r+r^2)=2s-1=1+2zrs=1/(1-2zs)$ and $s=(r+1)/2=1+zrs=1/(1-zr)=1+z(-s+2s^2)$.

Since $s=\dfrac{1}{1-zr}=\dfrac{1}{1-z}\circ(zr)$, and $\left(\dfrac{1}{1-z},\dfrac{z}{1-z}\right)$ is a pseudo-involution, it follows from \cite[Theorem 44]{BS} that the pseudo-involutory companion of $g=s$ is
\[
f=\widehat{(zr)}\circ\frac{z}{1-z}\circ(zr)=\frac{z(1+z)}{1-z}\circ\frac{zr}{1-zr}=\frac{z(1+z)}{1-z}\circ(zrs)
=\frac{zrs(1+zrs)}{1-zrs}=\frac{zrs^2}{2-s}.
\]

Let $g=s$ in Theorem~\ref{thm:bf-ogf}, then $\gamma(z)=-z+2z^2=z(-1+2z)$, so
\[
\eta\left(\frac{(z-1)^2}{z}\right)=\gamma(z)\gamma\left(\frac{1}{z}\right)=z(-1+2z)\cdot\frac{1}{z}\left(-1+\frac{2}{z}\right)=5-2z-\frac{2}{z}=1-2\frac{(z-1)^2}{z},
\]
so $\eta(z)=1-2z$, and therefore, $\overline{\left(\dfrac{z}{\eta}\right)}=\overline{\left(\dfrac{z}{1-2z}\right)}=\dfrac{z}{1+2z}$. Similarly,
\[
H\left(\frac{(z-1)^2}{z}\right)=\frac{\gamma(z)-z\gamma\left(\frac{1}{z}\right)}{z-1}=\frac{-z+2z^2+1-\frac{2}{z}}{z-1}=2z+1+\frac{2}{z}=5+2\frac{(z-1)^2}{z},
\]
so $H(z)=5+2z$, and thus
\[
B_f=H\circ\overline{\left(\frac{z}{\eta}\right)}=(5+2z)\circ\frac{z}{1+2z}=5+\frac{2z}{1+2z}=\frac{5+12z}{1+2z}.
\]
In other words, the B-sequence $(b_n)_{n\ge 0}$ of $f$ is given by $b_0=5$ and $b_n=(-1)^{n-1}2^n$ for $n\ge 1$.
\end{example}

\begin{theorem} \label{thm:bf-egf}
Let $g$ be a function that satisfies Equation \eqref{eq:gamma-egf}, and let $f=z\frac{\gamma(g)}{\gamma(1/g)}$. Then $B_f=E\circ\overline{\left(\frac{z}{\varepsilon}\right)}\circ\sqrt{z}$, where
\[
\varepsilon=\varepsilon(z)=\sqrt{\gamma(e^z)\gamma(e^{-z})}, \qquad E(z)=\frac{\gamma(e^z)-\gamma(e^{-z})}{z}.
\]
\end{theorem}

\begin{proof}
Since $g=e^{z\gamma(g)}$, we have
\[
(\log g)^2=z^2\gamma(g)^2=z^2\frac{\gamma(g)}{\gamma(1/g)}\gamma(g)\gamma(1/g)=(zf)\cdot((\gamma(z)\gamma(1/z))\circ g)
\]
Let $\varepsilon=\varepsilon(z)=\sqrt{\gamma(e^z)\gamma(e^{-z})}$, then $\gamma(z)\gamma(1/z)=(\varepsilon(\log z))^2$, so
\[
(\log g)^2=(zf)\cdot(\varepsilon(\log g))^2,
\]
and, therefore,
\[
zf=\frac{(\log g)^2}{(\varepsilon(\log g))^2}=z^2\circ\frac{z}{\varepsilon}\circ(\log g).
\]
Note that $\varepsilon(0)=\gamma(1)^2\ne 0$, so $\frac{z}{\varepsilon}$ is invertible as a formal power series, and thus
\[
\log g = \overline{\left(\frac{z}{\varepsilon}\right)}\circ\sqrt{z}\circ (zf).
\]
On the other hand, we have 
\[
B_f(zf)=\frac{f-z}{zf}=\frac{\frac{f}{z}-1}{f}=\frac{\gamma(g)-\gamma\left(\frac{1}{g}\right)}{z\gamma(g)}=\frac{\gamma(g)-\gamma\left(\frac{1}{g}\right)}{\log g}=\frac{\gamma(e^z)-\gamma(e^{-z})}{z}\circ(\log g).
\]
Now letting $E=E(z)=\frac{\gamma(e^z)-\gamma(e^{-z})}{z}$ and combining the preceding equalities, we obtain
\[
B_f=E\circ\overline{\left(\frac{z}{\varepsilon}\right)}\circ\sqrt{z},
\]
as claimed. 
\end{proof}

Note that the power series expansion of $\gamma(e^z)-\gamma(e^{-z})$ has no constant term, and thus, the power series expansion of $E(z)$ has only nonnegative terms. In particular, \[
E(0)=\lim_{z\to 0}{\frac{\gamma(e^z)-\gamma(e^{-z})}{z}}=\left.\frac{d}{dz}(\gamma(e^z)-\gamma(e^{-z}))\right\vert_{z=0}=2\gamma'(1).
\]

\begin{example} \label{ex:labeled-trees}
\emph{(Rooted labeled trees)} Let $T=T(z)$ be the exponential generating functions for the class of labeled trees on vertices $\{0,1,\dots,n\}$, rooted at $0$ (see Examples 22 and 31 in~\cite{BS}). Then $T$ satisfies the equation $T=e^{zT}$, so $\gamma(z)=z$, palindromic of darga $1+1=2$, which implies that $f=zT^2$. To find $B_f$, we see that $\varepsilon=\sqrt{e^z\cdot e^{-z}}=1$, so $\overline{\left(\frac{z}{\varepsilon}\right)}=\overline{z}=z$ and $E=\frac{e^z-e^{-z}}{z}=2\frac{\sinh(z)}{z}$, and thus $B_f=\left(2\frac{\sinh(z)}{z}\right)\circ z\circ\sqrt{z}=2\frac{\sinh(\sqrt{z})}{\sqrt{z}}$, which agrees with~\cite[Example 31]{BS}.
\end{example}

\begin{example} \label{ex:labeled-trees-2colored}
\emph{(Rooted labeled trees with 2-colored leaves)} Let $S=S(z)$ be the exponential generating functions for the class of labeled trees on vertices $\{0,1,\dots,n\}$, rooted at $0$ with 2-colored trees, except the singleton tree is 1-colored (see Examples 23 and 32 in~\cite{BS}). Then $S$ satisfies the equation $S=e^{z(1+S)}$, so $\gamma(z)=1+z$, palindromic of darga $0+1=1$, which implies that $f=zS$. To find $B_f$, we see that $\varepsilon=\sqrt{(1+e^z)(1+e^{-z})}=\frac{1+e^z}{e^{z/2}}=2\cosh(z/2)$, so
\[
\overline{\left(\frac{z}{\varepsilon}\right)}=\overline{\left(\frac{z}{2\cosh(z/2)}\right)}=\overline{\left(\frac{z}{\cosh(z)}\circ\frac{z}{2}\right)}=2\overline{\left(\frac{z}{\cosh(z)}\right)}=2\overline{(z\sech(z))}
\] 
and 
\[
E=\frac{(1+e^z)-(1+e^{-z})}{z}=2\frac{\sinh(z)}{z},
\] 
and thus
\[
\begin{split}
B_f&=\left(2\frac{\sinh(z)}{z}\right)\circ\left(2\overline{(z\sech(z))}\right)\circ\sqrt{z}=\left(2\frac{\sinh(2z)}{2z}\right)\circ\overline{(z\sech(z))}\circ\sqrt{z}\\
&=\frac{2\sinh(z)\cosh(z)}{z}\circ\overline{(z\sech(z))}\circ\sqrt{z}=\frac{2\sinh(z)}{z\sech(z)}\circ\overline{(z\sech(z))}\circ\sqrt{z}\\
&=2\frac{\sinh(z)\circ\overline{(z\sech(z))}}{z}\circ\sqrt{z},
\end{split}
\]
which agrees with~\cite[Example 32]{BS}.
\end{example}

A substitution $g(z)\mapsto g(kz)$ yields another useful general result.

\begin{theorem} \label{thm:bf-k}
Let $(g,f)$ and $(G,F)$ be pseudo-involutions such that $G(z)=g(kz)$ for some constant $k\ne 0$. Then
\[
F(z)=\frac{1}{k}f(kz), \quad B_F(z)=k B_f(k^2 z).
\]
\end{theorem}

\begin{proof}
Note that $kF(z)=f(kz)$ and $(zf)\circ(kz)=kzf(kz)=k^2 zF=(k^2z)\circ(zF)$. Then 
\[
G(-F(z))=g\circ(kz)\circ\left(-\frac{1}{k}f(kz)\right)=g\circ(-f(kz))=\frac{1}{g(kz)}=\frac{1}{G(z)},\\
\]
and
\[
\begin{split}
F-z&=\frac{1}{k}f(kz)-z=\frac{1}{k}(f-z)\circ(kz)=\frac{1}{k}(zB_f)\circ(zf)\circ(kz)\\
&=\frac{1}{k}(zB_f)\circ(k^2 z)\circ(zF)=\left(kz B_f(k^2 z)\right) \circ (zF). \qedhere
\end{split}
\]
\end{proof}

The last equation is also equivalent to $zB_F^2=k^2zB_f^2(k^2z)=(zB_f^2)\circ(k^2z)$.

\section{$B$-function given a Laurent polynomial $\gamma$} \label{sec:poly-gamma}

In this section, we will find the $B$-function for a Riordan pseudo-involution $(g,f)$ where $g$ satisfies Equation \eqref{eq:gamma-ogf} for some \emph{Laurent polynomial} $\gamma$ (with $\gamma(1)\ne 0$ unless $\gamma=0$). In other words, $\gamma(z)=\sum_{j\in\mathbb{Z}} c_j z^j$ for some constants $c_j$ ($j\in\mathbb{Z}$), only finitely many of which are nonzero, and $\sum_{j\in\mathbb{Z}} c_j\ne 0$ if $\gamma\ne 0$. Equivalently, either $\gamma(z)=0$ or $\gamma(z)=z^\ell \gamma_0(z)$ for some $\ell\in\mathbb{Z}$ and some polynomial $\gamma_0=\gamma_0(z)$ with $\gamma_0(0)\ne 0$, $\gamma_0(1)\ne 0$. When $\gamma=0$, we have $g=1$, $f=z$, $B_f=0$, so from now on we may assume that $\gamma(1)\ne 0$.

In what follows, we will make extensive use of two families of polynomials defined in \cite{BS} that are related to Chebyshev polynomials. We define and briefly recall their properties as given in Sections 4.1 and 4.3 of \cite{BS}.  Other connections between Chebyshev polynomials and Riordan involutions, different from those in this paper, were previously considered by Barry~\cite{B2}.

The first family, $(p_n(z))_{n\ge 0}$, is defined as follows. For each integer $l\ge 0$, let
\begin{equation} \label{eq:pk_cheb}
\begin{split}
p_{2l}(z) &= \left(U_l\left(\frac{z + 2}{2}\right) + U_{l-1}\left(\frac{z + 2}{2}\right)\right)^2,\\
p_{2l+1}(z) &= (z+4)\left(U_l\left(\frac{z + 2}{2}\right)\right)^2,
\end{split}
\end{equation}
where $U_l(z)$, defined by $U_l(\cos\theta)=\dfrac{\sin((l+1)\theta)}{\sin\theta}$, is the Chebyshev polynomial of the second kind. Also, define $p_{-1}(z)=0$, which agrees with \eqref{eq:pk_cheb} when $l=-1$. Then 
%\begin{equation} \label{eq:poly-p}
$
p_n(z)=\sum_{k=0}^{n}{d_{n,k}z^k},
$
%\end{equation}
where 
\[
d_{n,k}=\frac{n+1}{k+1}\binom{n+k+1}{2k+1}=\binom{n+k+2}{2k+2}+\binom{n+k+1}{2k+2},
\]
and 
\[
[d_{n,k}]_{n,k\ge 0}=\left(\frac{1+z}{(1-z)^3},\frac{z}{(1-z)^2}\right)
\]
is the Riordan array \seqnum{A156308}, whose first few rows are given by
\[
\begin{bmatrix}
1 & 0 & 0 & 0 & 0\\
4 & 1 & 0 & 0 & 0\\
9 & 6 & 1 & 0 & 0\\
16 & 20 & 8 & 1 & 0\\
25 & 50 & 35 & 10 & 1
\end{bmatrix}.
\]
Moreover, for any integer $n\ge 0$, and formal variables $u$ and $v$,
\begin{equation} \label{eq:p-dnk}
(zp_n(z))\circ\frac{(u-v)^2}{uv}=\frac{(u^{n+1}-v^{n+1})^2}{(uv)^{n+1}}.
\end{equation}

We define the second family, $(P_n(z))_{n\ge 0}$, by 
\begin{equation} \label{eq:p-def}
P_n(z)=U_n\left(\frac{z + 2}{2}\right) + U_{n-1}\left(\frac{z + 2}{2}\right).
\end{equation}
Then, for all $n\ge 0$,
\begin{equation} \label{eq:p-root}
P_n(z)=p_n(z)-p_{n-1}(z)=\sqrt{p_{2n}(z)}.
\end{equation}
Let
%\begin{equation} \label{eq:pl-def}
$
P_n(z)=\sum_{k=0}^{n}{a_{n,k}z^k},
$
%\end{equation}
then
\begin{equation} \label{eq:alj-def}
a_{n,k}=\frac{2n+1}{2k+1}\binom{n+k}{2k}=\binom{n+k+1}{2k+1}+\binom{n+k}{2k+1}, \quad n,k\ge 0,
\end{equation}
and
\[
[a_{n,k}]_{n,k\ge 0}=\left(\frac{1+z}{(1-z)^{2}},\frac{z}{(1-z)^{2}}\right)
\]
is the Riordan array \seqnum{A111125}, which begins with
\[
\begin{bmatrix}
1 & 0 & 0 & 0 & 0\\ 
3 & 1 & 0 & 0 & 0\\ 
5 & 5 & 1 & 0 & 0\\ 
7 & 14 & 7 & 1 & 0\\ 
9 & 30 & 27 & 9 & 1\\
\end{bmatrix}.
\]
Moreover, the array $[a_{n,k}(uv)^{n-k}]_{n,k\ge 0}$ satisfies the polynomial identity
\begin{equation} \label{eq:aij}
\left[a_{n,k}(uv)^{n-k}\right]_{n,k\ge 0}\left[(u-v)^{2k+1}\right]_{k\ge 0}^T=\left[u^{2n+1}-v^{2n+1}\right]_{l\ge 0}^T
\end{equation}
(where the superscript ${}^T$ denotes the transpose), or equivalently, for any $n\ge 0$,
\begin{equation} \label{eq:P-ank}
P_n\left(\frac{(u-v)^2}{uv}\right)=\frac{u^{2n+1}-v^{2n+1}}{(u-v)(uv)^n}.
\end{equation}

A simple verification shows that we can extend the range of the index $n$ in $p_n$ and $P_n$ to $\mathbb{Z}$ to agree with properties \eqref{eq:alj-def} and \eqref{eq:P-ank} by setting
\begin{equation} \label{eq:pp-neg}
p_{-1}(z)=0, \qquad p_{-n}(z)=p_{n-2}(z) \ \text{ for } n\ge 2, \qquad P_{-n}(z)=-P_{n-1}(z) \ \text{ for } n\ge 1.
\end{equation}
The second of the three definitions in \eqref{eq:pp-neg} can be more symmetrically restated as 
\[
p_{-n-1}(z)=p_{n-1}(z)\text{ for } n\ge 0.
\]
This also follows from the similar extension of $U_n(z)$ to all indices $n\in\mathbb{Z}$, since
\[
U_{-n-1}(\cos\theta)=\frac{\sin(-n\theta)}{\sin\theta}=-\frac{\sin(n\theta)}{\sin\theta}=-U_{n-1}(\cos\theta), \quad n\ge 0,
\]
so $U_{-n-1}(z)=-U_{n-1}(z)$ and $U_{-1}(z)=0$.

\bigskip

\begin{theorem} \label{thm:Laurent-poly}
Let $g$, $\gamma$, $f$, $\eta$, and $H$ be as in Theorem~\ref{thm:bf-ogf}, and suppose that $\gamma(z)=\sum_{j\in\mathbb{Z}} c_j z^j$ is a Laurent polynomial $\gamma(1)\ne 0$, i.e. only finitely many of the constants $c_j$ ($j\in\mathbb{Z}$) are nonzero and $\sum_{j\in\mathbb{Z}} c_j\ne 0$. Then
\begin{equation} \label{eq:eta-Laurent}
\eta(z)=\gamma(1)^2+z\sum_{n=1}^{\infty}\left(\sum_{j\in\mathbb{Z}}c_{j}c_{n+j}\right)p_{n-1}(z)
\end{equation}
and
\begin{equation} \label{eq:H-Laurent}
H(z)=\sum_{n\in\mathbb{Z}}c_nP_{n-1}(z).
\end{equation}
\end{theorem}

\begin{proof}
Since the product
\[
\gamma(z)\gamma\left(\frac{1}{z}\right)=\left(\sum_{j\in\mathbb{Z}} c_j z^j\right)\left(\sum_{j\in\mathbb{Z}} c_j z^{-j}\right),
\]
is invariant with respect to the substitution $z\mapsto 1/z$, it follows that for each $k\ge 1$, its coefficients at $z^k$ and $z^{-k}$ are equal. Therefore, the function $\eta=\eta(z)$ as defined in Theorem~\ref{thm:bf-ogf} satisfies
\[
\begin{split}
\eta\left(\frac{(z-1)^2}{z}\right)=\gamma(z)\gamma\left(\frac{1}{z}\right)&=\left(\sum_{j\in\mathbb{Z}}c_j^2\right)+\sum_{n=1}^{\infty}\left(\sum_{j\in\mathbb{Z}}c_{j}c_{n+j}\right)(z^n+z^{-n})\\
&=\left(\sum_{j\in\mathbb{Z}}c_j\right)^2+\sum_{n=1}^{\infty}\left(\sum_{j\in\mathbb{Z}}c_{j}c_{n+j}\right)(z^n+z^{-n}-2)\\
&=\gamma(1)^2+\sum_{n=1}^{\infty}\left(\sum_{j\in\mathbb{Z}}c_{j}c_{n+j}\right)(zp_{n-1})\circ\frac{(z-1)^2}{z},
\end{split}
\]
or equivalently,
\[
\eta(z)=\gamma(1)^2+z\sum_{n=1}^{\infty}\left(\sum_{j\in\mathbb{Z}}c_{j}c_{n+j}\right)p_{n-1}(z).
\]
Similarly, letting $u=\sqrt{z}$ and $v=\frac{1}{\sqrt{z}}$ in \eqref{eq:P-ank}, we obtain
\[
\begin{split}
H\left(\frac{(z-1)^2}{z}\right)&=\frac{\gamma(z)-z\gamma\left(\frac{1}{z}\right)}{z-1}=\frac{\frac{1}{\sqrt{z}}\gamma(z)-\sqrt{z}\gamma\left(\frac{1}{z}\right)}{\sqrt{z}-\frac{1}{\sqrt{z}}}\\
&=\sum_{n\in\mathbb{Z}}c_n\frac{z^{n-\frac{1}{2}}-z^{-n+\frac{1}{2}}}{\sqrt{z}-\frac{1}{\sqrt{z}}}
=\sum_{n\in\mathbb{Z}}c_nP_{n-1}\left(\left(\sqrt{z}-\frac{1}{\sqrt{z}}\right)^2\right)
=\sum_{n\in\mathbb{Z}}c_nP_{n-1}\left(\frac{(z-1)^2}{z}\right),
\end{split}
\]
so 
\[
H(z)=\sum_{n\in\mathbb{Z}}c_nP_{n-1}(z).
\]
Equivalently, if we define a linear functional $L$ on the space of Laurent series in a formal variable $P$ so that $L(P^n)=P_n$ for all $n\in\mathbb{Z}$, then
\[
H=L\left(\frac{\gamma(P)}{P}\right). \qedhere
\] 
\end{proof}

\begin{remark} \label{rem:gamma0}
If $\gamma(z)\ne0$, then we can write $\gamma(z)=z^\ell \gamma_0(z)$, where $\ell\in\mathbb{Z}$ is the minimum degree of $\gamma$, and $\gamma_0=\gamma_0(z)$ is a polynomial with $\gamma_0(0)\ne 0$. Notice that $\gamma(z)\gamma(1/z)=\gamma_0(z)\gamma_0(1/z)$, so we may let $c'_j=c_{j+\ell}$ for $j\ge 0$ and use $\gamma_0(z)=\sum_{j=0}^{\infty} c'_j z^j$ instead of $\gamma(z)$ to calculate $\eta(z)$ more efficiently. Indeed, the same calculations for $\gamma_0$ show that
\begin{equation} \label{eq:eta-gamma0}
\eta(z)=\gamma_0(1)^2+z\sum_{n=1}^{\infty}\left(\sum_{j\ge 0}c'_{j}c'_{n+j}\right)p_{n-1}(z).
\end{equation}
\end{remark}

\begin{example} \label{ex:little-sch-laurent} 
\emph{(Little Schr\"{o}der numbers)}
Let $g=s$ as in Example~\ref{ex:little-sch}, and let $f$ be the pseudo-involutory companion of $g$. Then we have $\gamma(z)=-z+2z^2=z(-1+2z)$, so $\gamma_0(z)=-1+2z$, and thus $\ell=1$ and the nonzero coefficients are $c'_0=c_1=-1$ and $c'_1=c_2=2$. Recall also that $p_0(z)=1$, $P_0(z)=1$, and $P_1(z)=3+z$. Therefore,
\[
\eta(z)=(c'_0+c'_1)^2+z(c'_0c'_1)p_0(z)=(-1+2)^2+z\cdot(-1)\cdot2\cdot1=1-2z
\]
and
\[
H(z)=c_1P_0(z)+c_2P_1(z)=(-1)\cdot1+2(3+z)=5+2z,
\]
as expected from Example~\ref{ex:little-sch}.
\end{example}

\begin{example} \label{ex:double-ext-motzkin-laurent}
\emph{(Extended doubled Motzkin numbers)}
Let $m=m(z)=1+zm+z^2m^2$ be the ordinary generating function for the Motzkin numbers \seqnum{A001006} and let $\tilde{m}=1+zm=1+z(1-\tilde{m}+\tilde{m}^2)$ be the generating function for the extended Motzkin sequence (\seqnum{A086246} without the initial $0$). From \cite[Example 49]{BS}, we have that pseudo-involutory companion of $g=2\tilde{m}-1=1+2zm$ (whose coefficient sequence is \seqnum{A007971} without the initial $0$) is
\[
f=\frac{z(2\tilde{m}-1)}{1+2z\tilde{m}}=\frac{z(2\tilde{m}-1)}{1+2\tilde{m}\frac{\tilde{m}-1}{1-\tilde{m}+\tilde{m}^2}}
=z(2\tilde{m}-1)\frac{1-\tilde{m}+\tilde{m}^2}{1-3\tilde{m}+3\tilde{m}^2}=\frac{(2\tilde{m}-1)(\tilde{m}-1)}{1-3\tilde{m}+3\tilde{m}^2}=\frac{1-3\tilde{m}+2\tilde{m}^2}{1-3\tilde{m}+3\tilde{m}^2},
\]
with the coefficient sequence \seqnum{A348189}. Since
\[
g=2\tilde{m}-1=1+2z(1-\tilde{m}+\tilde{m}^2)=1+2z\left(1-\frac{1+g}{2}+\left(\frac{1+g}{2}\right)^2\right)=1+z\frac{3+g^2}{2},
\]
it follows that $\gamma(z)=\dfrac{3+z^2}{2}$, so we only need $c_0=3/2$, $c_1=0$, and $c_2=1/2$. Therefore,
\[
\eta(z)=\gamma(1)^2+z((c_0c_1+c_1c_2)p_0(z)+(c_0c_2)p_1(z))=2^2+z\left(0\cdot1+\frac{3}{4}(4+z)\right)=4+3z+\frac{3}{4}z^2,
\]
so
\[
\frac{z}{\eta}=\frac{4z}{16+12z+3z^2}=\frac{z}{1+3z+3z^2}\circ\frac{z}{4}
=\frac{z}{1+3z}\circ\frac{z}{1+3z^2}\circ\frac{z}{4},
\]
and thus,
\[
\overline{\left(\frac{z}{\eta}\right)}=(4z)\circ(zC(3z^2))\circ\frac{z}{1-3z}.
\]
Similarly,
\[
H(z)=c_0P_{-1}(z)+c_1P_0(z)+c_2P_1(z)=\frac{3}{2}\cdot(-1)+0\cdot1+\frac{1}{2}(3+z)=\frac{z}{2},
\]
so
\[
B_f=H\circ\overline{\left(\frac{z}{\eta}\right)}
=(2z)\circ(zC(3z^2))\circ\frac{z}{1-3z}=\frac{2z}{1-3z}C\left(\frac{3z^2}{(1-3z)^2}\right)
=\frac{1-3z-\sqrt{1-6z-3z^2}}{3z},
\]
whose coefficient sequence is $2\cdot\text{\seqnum{A107264}}$ with $0$ prepended.
\end{example}

\begin{example} \label{ex:quad-laurent}
We can generalize the results of Examples~\ref{ex:little-sch}, \ref{ex:little-sch-laurent}, and~\ref{ex:double-ext-motzkin-laurent} to explicitly find $B_f$ for the pseudo-involutory companion $f$ of any function $g$ that satisfies the equation $g=1+z\gamma(g)$ for a quadratic polynomial $\gamma$ such that $\gamma(1)\ne 0$. Let $\gamma(z)=a+bz+cz^2$ such that $a+b+c\ne 0$. Then
\[
\begin{split}
\eta=\eta(z)&=\gamma(1)^2+z((ab+bc)p_0(z)+(ac)p_1(z))\\
&=(a+b+c)^2+z((ab+bc)\cdot 1+ac(4+z))\\
&=(a+b+c)^2+(ab+bc+4ac)z+acz^2,
\end{split}
\]
so
\[
\begin{split}
\frac{z}{\eta}&=\frac{z}{(a+b+c)^2+(ab+bc+4ac)z+acz^2}\\
&=\frac{z}{1+(ab+bc+4ac)z+ac(a+b+c)^2z^2}\circ\frac{z}{(a+b+c)^2}\\
&=\frac{z}{1+(ab+bc+4ac)z}\circ\frac{z}{1+ac(a+b+c)^2z^2}\circ\frac{z}{(a+b+c)^2},
\end{split}
\]
and thus,
\[
\begin{split}
\overline{\left(\frac{z}{\eta}\right)}&=((a+b+c)^2 z)\circ\left(zC\left(ac(a+b+c)^2z^2\right)\right)\circ\frac{z}{1-(ab+bc+4ac)z}\\
&=((a+b+c)z)\circ\left((a+b+c)zC\left(ac(a+b+c)^2z^2\right)\right)\circ\frac{z}{1-(ab+bc+4ac)z}\\
&=\left((a+b+c)zC\left(acz^2\right)\right)\circ\frac{(a+b+c)z}{1-(ab+bc+4ac)z}.
\end{split}
\]
Similarly,
\[
H(z)=aP_{-1}(z)+bP_0(z)+cP_1(z)=a\cdot(-1)+b\cdot1+c(3+z)=(-a+b+3c)+cz,
\]
so
\begin{equation} \label{eq:bf-quad-laurent}
\begin{split}
B_f=H\circ\overline{\left(\frac{z}{\eta}\right)}
&=((-a+b+3c)+cz)\circ\left((a+b+c)zC\left(acz^2\right)\right)\circ\frac{(a+b+c)z}{1-(ab+bc+4ac)z}\\
&=(-a+b+3c)+\left((a+b+c)czC\left(acz^2\right)\right)\circ\frac{(a+b+c)z}{1-(ab+bc+4ac)z}.
\end{split}
\end{equation}
Note that $C(0)=1$, so when $a=0$, we have
\[
B_f=(b+3c)+((b+c)cz)\circ\frac{(b+c)z}{1-bcz}=(b+3c)+\frac{(b+c)^2cz}{1-bcz}=\frac{(b+3c)-(b-c)cz}{1-bcz},
\]
as in Examples~\ref{ex:little-sch} and~\ref{ex:little-sch-laurent}. Likewise, when $c=0$ (that is, $\gamma(z)=a+bz$ is linear), Equation~\eqref{eq:bf-quad-laurent} reduces to $B_f=b-a$, since in this case $(g,f)=\left(\frac{1+az}{1-bz},\frac{z}{1-(b-a)z}\right)$, as in~\cite[Example 3]{BS}.

Note that when $a=0$, $b\ne 0$, $c\ne 0$, it follows that $B_f$ is a fractional linear transformation of the form $\frac{A-Cz}{1+Bz}$, previously considered by Barry~\cite{B1}.
\end{example}

Finally, consider what happens when $\gamma(1)=a+b+c=0$. In that case, either $c=0$ and $\gamma(z)=b(z-1)$, or $c\ne 0$ and $\gamma(z)=c(z-1)(z-\frac{a}{c})$. In the former case, we have $g=1$, $f=z$, and $B_f=0$, which agrees with Equation~\eqref{eq:bf-quad-laurent}. In the latter case, we have $g-1=cz(g-1)(g-\frac{a}{c})$, and thus, either $g=1$, $f=z$, and $B_f=0$, as before, or $g=\frac{a}{c}+\frac{1}{cz}$, which is not a power series in $z$.

\section{B-function given a rational function $g$} \label{sec:rat-g}

In this section, we consider somewhat of an inverse special case. Suppose $g=g(z)$ is a rational function, and $f$ is the pseudo-involutory companion of $g$. How can we find the B-function $B_f$ in this case?

Let us write $g(z)=\frac{p(z)}{q(z)}$ for some relatively prime polynomials $p$ and $q$ such that $p(0)=1$ and $q(0)=1$. Then the fact that $(g,f)$ is a pseudo-involution implies that $g(-f)=1/g$, that is $\frac{p(-f)}{q(-f)}=\frac{q(z)}{p(z)}$, or equivalently, $p(z)p(-f)-q(z)q(-f)=0$. Let $u$ and $v$ be two formal variables and consider the function
\[
R(u,v)=p(u)p(v)-q(u)q(v).
\]
$R$ is a polynomial symmetric in $u$ and $v$, so we can write $R(u,v)=S(u+v,uv)$ for some bivariate polynomial $S=S(x,y)$. Note that $S(0,0)=R(0,0)=0$. Then
\[
S(z-f,zf)=R(z,-f)=p(z)p(-f)-q(z)q(-f)=0.
\]
But $z-f=-(zB_f)\circ(zf)$, so we have
\[
S(-zB_f,z)\circ(zf)=S(z-f,zf)=0.
\]
Therefore, $zB_f(z)$ is a power series solution of $S(-x,z)=0$ in $x$.

\begin{example} \label{ex:bf-fib}
\emph{(Fibonacci numbers)}
As in~\cite[Example 50]{BS}, let $g=g(z)=\frac{1}{1-z-z^2}$ be the ordinary generating function for the Fibonacci sequence (starting at $1,1$), and let $f=f(z)$ be its pseudo-involutory companion. Since $g=\frac{1}{1-z}\circ(z+z^2)$, and $\frac{z}{1-z}$ is the  pseudo-involutory companion of $\frac{1}{1-z}$, we have
\[
f=\widehat{(z+z^2)}\circ\frac{z}{1-z}\circ(z+z^2)=(zC)\circ\frac{z+z^2}{1-z-z^2}=(zC)\circ(g-1)=\frac{1}{2}\left(1-\sqrt{\frac{1-5z-5z^2}{1-z-z^2}}\right).
\]
The equation $g(-f)=1/g$ is equivalent to
\[
\frac{1}{1+f-f^2}=1-z-z^2,
\]
i.e.
\[
(1-z-z^2)(1+f-f^2)=1,
\]
or
\[
f-z-f^2-zf-z^2+zf^2-z^2f+z^2f^2=0,
\]
which can be rewritten as
\[
(f-z)-(f-z)^2-3zf+zf(f-z)+(zf)^2=0.
\]
Since $f-z=(zB_f)\circ(zf)$, it follows that
\[
zB_f-(zB_f)^2-3z+z(zB_f)+z^2=0,
\]
or equivalently,
\[
zB_f^2-(1+z)B_f+(3-z)=0.
\]
Since $B_f$ is the power series solution of this equation, it follows that
\[
B_f=\frac{1+z-\sqrt{(1+z)^2-4z(3-z)}}{2z}=\frac{1+z-\sqrt{1-10z+5z^2}}{2z},
\]
which was conjectured and verified in~\cite[Example 50]{BS} by substituting $f$ and $B_f$ into $f-z=(zB_f)\circ(zf)$.
\end{example}

To generalize this example, we define a third family of polynomials, $(Q_n(z))_{n\ge 0}$, by 
\begin{equation} \label{eq:q-def}
Q_0(z)=1, \qquad Q_n(z)=2T_n\left(\frac{z + 2}{2}\right), \quad n\ge 1,
\end{equation}
where $T_n(z)$ is the Chebyshev polynomial of the first kind, defined by $T_n(\cos\theta)=\cos(n\theta)$. Then
\begin{equation} \label{eq:p-q}
Q_n(z)=P_n(z)-P_{n-1}(z) \qquad \text{for } n\ge 1,
\end{equation}
where $P_n(z)$ is as defined in \eqref{eq:p-def}. Let $Q_n(z)=\sum_{k=0}^{n}{b_{n,k}z^k}$,
then 
\begin{equation} \label{eq:bij-def}
b_{n,k}=\binom{n+k}{2k}+\binom{n+k-1}{2k}, \quad n,k\ge 0,
\end{equation}
and
\[
[b_{n,k}]_{n,k\ge 0}=\left(\frac{1+z}{1-z},\frac{z}{(1-z)^{2}}\right)
\]
is unsigned version of the Riordan array \seqnum{A110162}, which begins with
\[
\begin{bmatrix}
1 & 0 & 0 & 0 & 0\\ 
2 & 1 & 0 & 0 & 0\\ 
2 & 4 & 1 & 0 & 0\\ 
2 & 9 & 6 & 1 & 0\\ 
2 & 16 & 20 & 8 & 1\\
\end{bmatrix}.
\]
Moreover, the array $[b_{n,k}(uv)^{n-k}]_{n,k\ge 0}$ satisfies the polynomial identity
\begin{equation} \label{eq:bij}
\left[b_{n,k}(uv)^{n-k}\right]_{n,k\ge 0}\left[(u-v)^{2k}\right]_{k\ge 0}^T=\left[u^{2n}+v^{2n} - [n=0]\right]_{l\ge 0}^T
\end{equation}
(where the superscript ${}^T$ denotes the transpose, and $[n=0]$ is the Iverson bracket), or equivalently,
\begin{equation} \label{eq:Q-bnk}
Q_0(z)=1, \qquad Q_n\!\left(\frac{(u-v)^2}{uv}\right)=\frac{u^{2n}+v^{2n}}{(uv)^n} \quad \text{for } n\ge 1.
\end{equation}

Since both products in $p(z)p(-f)-q(z)q(-f)$ have similar form, it is enough to understand how to find a polynomial $H_1(z,y)$ such that $H_1(z,B_f)\circ(zf)=p(z)p(-f)$. Then we can similarly find $H_2(z,y)$ such that $H_2(z,B_f)\circ(zf)=q(z)q(-f)$ and then find the power series solution $y=B_f(z)$ for $H_1(z,y)=H_2(z,y)$. 

Suppose that $p(z)=\sum_{n\ge 0} c_nz^n$, where $c_0=1$ and only finitely many coefficients $c_n$ are nonzero. Then
\[
\begin{split}
&p(z)p(-f)=\left(\sum_{n\ge 0} c_nz^n\right)\left(\sum_{n\ge 0} c_n(-f)^n\right)\\
&=\sum_{n\ge 0}c_n(-zf)^n\left(c_n+\sum_{k\ge 1}c_{n+k}\left(z^k+(-f)^k\right)\right)\\
&=\sum_{n\ge 0}c_n(-zf)^n\left(c_n+\sum_{k\ge 1}c_{n+2k}\left(z^{2k}+f^{2k}\right)+\sum_{k\ge 0}c_{n+2k+1}\left(z^{2k+1}-f^{2k+1}\right)\right)\\
&=\sum_{n\ge 0}c_n(-zf)^n\left(\sum_{k\ge 0}c_{n+2k}(zf)^kQ_k\left(\frac{(f-z)^2}{zf}\right)+\sum_{k\ge 0}c_{n+2k+1}(zf)^k(z-f)P_k\left(\frac{(f-z)^2}{zf}\right)\right)\\
&=\sum_{n\ge 0}c_n(-z)^n\left(\sum_{k\ge 0}c_{n+2k}z^kQ_k(zB_f^2)-(zB_f)\sum_{k\ge 0}c_{n+2k+1}z^kP_k(zB_f^2)\right)\circ(zf)\\
&=\left(\sum_{n,k\ge 0}(-1)^nc_nz^{n+k}\left(c_{n+2k}Q_k(zB_f^2)-c_{n+2k+1}(zB_f)P_k(zB_f^2)\right)\right)\circ(zf),
\end{split}
\]
so
\begin{equation} \label{eq:rat-p-cheb}
H_1(z,B_f)=\sum_{n,k\ge 0}(-1)^nc_nz^{n+k}\left(c_{n+2k}Q_k(zB_f^2)-c_{n+2k+1}(zB_f)P_k(zB_f^2)\right),
\end{equation}
where we use the facts that $f-z=(zB_f)\circ(zf)$ and $\frac{(f-z)^2}{zf}=(zB_f^2)\circ(zf)$.

To make this formula more compact, combine the polynomial sequences $(Q_n)_{n\ge 0}$ and $(P_n)_{n\ge 0}$ into a single family $(R_n)_{n\ge 0}$ defined by
\begin{equation} \label{eq:r-def}
R_{2n}(z)=Q_n(z^2), \qquad R_{2n+1}(z)=zP_n(z^2), \qquad n\ge 0.
\end{equation}
Let $R_n(z)=\sum_{n\ge 0}\alpha_{n,k}z^k$, then $R_0(z)=1$ and
\begin{equation} \label{eq:r-uv}
R_n\left(\frac{u^2-v^2}{uv}\right)=\frac{(u^2)^n+(-v^2)^n}{(uv)^n}, \qquad n\ge 1.
\end{equation}
Moreover,
\begin{equation} \label{eq:alphaij-def}
\alpha_{n,k}=\binom{\frac{n+k}{2}}{k}+\binom{\frac{n+k}{2}-1}{k}, \qquad n,k\ge 0,
\end{equation}
where we set $\binom{a}{b}=0$ when $a\notin\mathbb{Z}_{\ge 0}$,  so
\[
[\alpha_{n,k}]_{n,k\ge 0}=\left(\frac{1+z^2}{1-z^2}\ , \frac{z}{1-z^2}\right),
\]
the Riordan array given by the unsigned version of \seqnum{A108045}, or equivalently, \seqnum{A114525} where the initial $2$ is replaced with $1$.
Then
\[
\begin{split}
&p(z)p(-f)=\left(\sum_{n\ge 0} c_nz^n\right)\left(\sum_{n\ge 0} c_n(-f)^n\right)\\
&=\sum_{n\ge 0}c_n(-zf)^n\left(c_n+\sum_{k\ge 1}c_{n+k}\left(z^k+(-f)^k\right)\right)\\
&=\sum_{n,k\ge 0}c_n c_{n+k}(-zf)^n(zf)^{k/2}R_k\left(\frac{z-f}{\sqrt{zf}}\right)\\
&=\left(\sum_{n,k\ge 0}c_nc_{n+k}(-z)^n z^{k/2}R_k\left(-\sqrt{z}B_f\right)\right)\circ(zf),
\end{split}
\]
and thus \eqref{eq:rat-p-cheb} simplifies to
\begin{equation} \label{eq:rat-r-cheb}
H_1(z,B_f)=\sum_{n,k\ge 0}(-1)^n c_n c_{n+k}z^{n+\frac{k}{2}}R_k\left(-\sqrt{z}B_f\right).
\end{equation}

We can collect all of the above into the following theorem.
\begin{theorem} \label{thm:poly-cheb}
Let $g=g(z)=p(z)/q(z)$ for some polynomials $p(z)$ and $q(z)$ such that $p(0)=1$ and $q(0)=1$. Let $f$ be the pseudo-involutory companion of $g$. Suppose that $p(z)=\sum_{n\ge 0}c_nz^n$ and $p(z)=\sum_{n\ge 0}c_nz^n$ and $q(z)=\sum_{n\ge 0}d_nz^n$, where only finitely many of $c_n$'s and $d_n$'s are nonzero. Then the B-function $B_f$ of $f$ is a power series solution of the bivariate polynomial equation
\begin{equation} \label{eq:poly-cheb}
\sum_{n,k\ge 0}(-1)^n (c_n c_{n+k}-d_n d_{n+k})z^{n+\frac{k}{2}}R_k\left(-\sqrt{z}B_f\right)=0,
\end{equation}
where $R_k(z)$, $k\ge 0$, are the polynomials defined by any one of the equations \eqref{eq:r-def}, \eqref{eq:r-uv}, or \eqref{eq:alphaij-def}.
\end{theorem}

\begin{example} \label{ex:fib-cheb}
Let us apply Equation~\eqref{eq:rat-r-cheb} in the case of $g(z)=\frac{1}{1-z-z^2}$ from Example~\ref{ex:bf-fib}. With the notation as above, we have $p(z)=1$ and $q(z)=1-z-z^2$, so that $p(z)p(-f)=1$ and thus $H_1(z,B_f)=1$, so we only need to determine $H_2(z,B_f)$. For $q(z)=1-z-z^2$, we have $c_0=1$, $c_1=-1$, $c_2=-1$, so we only need to use $R_0(z)=1$, $R_1(z)=z$, and $R_2(z)=2+z^2$. Thus,
\[
\begin{split}
H_2(z,B_f)&=(c_0^2-c_1^2z+c_2^2z^2)\cdot1+(c_0c_1\sqrt{z}-c_1c_2z\sqrt{z})\left(-\sqrt{z}B_f\right)+c_0c_2z\left(2+\left(-\sqrt{z}B_f\right)^2\right)\\
&=(1-z+z^2)+(-\sqrt{z}-z\sqrt{z})\left(-\sqrt{z}B_f\right)-z\left(2+zB_f^2\right)\\
&=1-z+z^2+(z+z^2)B_f-2z-z^2B_f^2.
\end{split}
\]
Therefore, $H_1(z,B_f)=H_2(z,B_f)$ means
\[
1=1-z+z^2+(z+z^2)B_f-2z-z^2B_f^2,
\]
or, equivalently,
\[
zB_f^2-(1+z)B_f+(3-z)=0,
\]
which agrees with Example~\ref{ex:bf-fib}.
\end{example}

\end{document}